\documentclass[12pt]{amsart} 
\usepackage{hyperref}
\usepackage{amssymb} 
\usepackage{colordvi} 
\usepackage{graphicx} 
\usepackage{vmargin} 
\setpapersize{USletter} 
\setmargrb{1in}{1in}{1in}{1in} 
\numberwithin{equation}{section} 
\newtheorem{theorem}{Theorem}[section] 
\newtheorem{proposition}[theorem]{Proposition} 
 
\newtheorem{corollary}[theorem]{Corollary} 
 
\newtheorem{defn}[theorem]{Definition}

\theoremstyle{definition}
\newtheorem{example}[theorem]{Example} 
\newtheorem{remark}[theorem]{Remark} 

\newenvironment{ex}{\begin{example}\rm}{\end{example}}

\newcommand{\CC}{\mathbb C}
\newcommand{\RR}{\mathbb R}
\newcommand{\PP}{\mathbb P}
\newcommand{\AAA}{\mathbb A}
\newcommand{\Osh}{\mathcal O}

\title{The Number of Eigenvalues of a Tensor}
\author{Dustin Cartwright and  Bernd Sturmfels}
\address{
Department of Mathematics\\
      University of California\\
      Berkeley, CA 94720,
      USA}
\email{\{dustin,bernd\}@math.berkeley.edu}

\begin{document}

\begin{abstract}
Eigenvectors of tensors, as studied recently in numerical multilinear algebra,
  correspond to  fixed points of self-maps of a projective space.
We determine the number of eigenvectors
and eigenvalues of a generic tensor, and
we show that 
the number of normalized eigenvalues of a symmetric tensor is always finite.
We also examine the characteristic polynomial and how 
its coefficients are related to
discriminants and resultants.
\end{abstract}

\maketitle

\section{Introduction}

In the current numerical analysis literature, considerable interest has arisen in
extending concepts that are familiar from linear algebra to the setting of 
multilinear algebra. One such familiar concept is that of an eigenvalue
of a square matrix. Several authors have explored definitions of eigenvalues
and eigenvectors for higher-dimensional tensors, and they have argued that these notions are
useful for wide range of applications \cite{Lim}.  Our approach in this
paper is based on the definition of 
{\em E-eigenvalues} of tensors introduced by Liqun Qi in \cite{NQWW,Qi}. 
Throughout this paper, eigenvalue will mean E-eigenvalue, as defined in
Definition~\ref{def:eig}.

We fix two positive integers $m$ and $n$, and we consider 
order-$m$ tensors $A = (a_{i_1 i_2 \cdots i_m})$ 
of format $\,n \times n \times
\cdots \times n\,$ with entries in the field of complex numbers $\CC$.

\begin{defn} \label{def:eig} \rm
Let $x$ be in $\CC^n$ and $A$ a tensor as above. Adopting the notation
introduced in \cite{NQWW, Qi},
we define $Ax^{m-1}$ to be the vector in $\CC^n$ whose
$j$-th coordinate is the scalar
\begin{equation}
\label{eq:Am}
(Ax^{m-1})_{j} \quad = \quad \sum_{i_2 = 1}^n \cdots \sum_{i_m=1}^n a_{j i_2 \cdots i_m} x_{i_2}\cdots x_{i_m}.
\end{equation}
If $\lambda$ is a complex number and $x \in \CC^n$ a non-zero vector
such that $A x^{m-1} = \lambda x$, then $\lambda$ is an
\emph{eigenvalue} of~$A$ and $x$ is an \emph{eigenvector} of $A$. We will refer to the
pair of $\lambda$ and~$x$ as an \emph{eigenpair}.
Two eigenpairs $(\lambda,x)$ and $(\lambda', x')$ of the same tensor $A$
are considered to be {\em equivalent}
if there exists a complex number $t \neq 0$ such that $t^{m-2}\lambda = \lambda'$ and $t
x = x'$.
\end{defn}

If $m=2$ then $Ax^1$ is the ordinary matrix-vector product, and Definition \ref{def:eig}
recovers the familiar eigenvalues and eigenvectors of a square matrix $A$.
 In that case, our notion of equivalence
amounts to rescaling the eigenvector, but the eigenvalue is uniquely determined.

For $m \geq 3$, Qi \cite{Qi} normalizes the
eigenvectors $x$ of $A$ by additionally requiring $x \cdot x = 1$.
When $x \cdot x$ is non-zero, this has the effect of choosing two distinguished
representatives, related by $\lambda' = (-1)^m \lambda$ and $x' = -x$, from each
equivalence class. In particular, when $m$ is even, the eigenvalue is uniquely
determined, and when $m$ is odd, it is determined up to sign. However,
since equivalence classes with $x \cdot x = 0$ are not allowed by Qi's
normalization, his definition does not strictly generalize the classical
eigenvalues of a matrix. We will call eigenvalues $\lambda$ with an eigenvector 
satisfying $x \cdot x
= 1$ \emph{normalized eigenvalues} of the tensor~$A$.

In Section~6 of~\cite{Qi}, Qi considers
an alternative normalization by requiring $x \cdot
\overline x = 1$, where $\overline x$ is the complex conjugate. This reduces the
equivalence classes from two real dimensions to one real dimension. One can
still get an equivalent eigenpair $(t^{m-2}\lambda, t x)$ for
any complex number $t$ with unit modulus.
Yet another normalization, based on the $p$-norm over the real numbers $\RR$,
was introduced by Lek-Heng Lim in his variational approach \cite{Lim}.

For most of this paper, we prefer not to choose any normalization whatsoever.
Instead, we depend on the notion of equivalence in
Definition~\ref{def:eig} in order to have a finite number of equivalence classes
of eigenpairs in the generic case. This equivalence is a generalization of the
usual ambiguity of eigenvectors of an $n {\times} n$-matrix $A$, which at best,
are only unique up to scaling.  The following theorem generalizes, from $m=2$ to
$m \geq 3$, the familiar linear algebra fact that an $n {\times} n$-matrix $A$
has precisely $n$ eigenvalues over the complex numbers.

\begin{theorem}\label{thm:count}
If a tensor $A$ has finitely many equivalence classes of eigenpairs over~$\CC$
then their number, counted with multiplicity, is equal to $\,((m-1)^n - 1)/(m-2)$.
  If the entries of $A$ are sufficiently generic,
then all multiplicities are equal to $1$, so there are exactly $\,((m-1)^n -
1)/(m-2)\,$ equivalence classes of eigenpairs.
\end{theorem}

For the case when the tensor order $m$ is even, the above formula was derived in
\cite[Theorem~3.4]{NQWW} by means of a detailed analysis of
the Macaulay matrix for the multivariate resultant.
For arbitrary $m$, it was stated as a conjecture in line 2 on page
1228 of \cite{NQWW}.
 We here present a short proof of this conjecture that is based on
 techniques from toric geometry \cite{CS,fulton}.
 
The rest of this paper is organized as follows. In Section~\ref{sec:count}, we
prove Theorem~\ref{thm:count}. In Section~\ref{sec:characteristic}, we
investigate the characteristic polynomial. Tensors with vanishing characteristic
polynomial are interpreted as singular tensors.
 In Section~\ref{sec:dynamics}, we
relate eigenvalues of tensors to dynamics on projective space. Finally, in
Section~\ref{sec:symmetric}, we specialize to the case of symmetric tensors.
We
show that, in that case, the set of normalized eigenvalues is always finite.

\section{Intersections in a Weighted Projective Space}\label{sec:count}

We shall formulate the problem of computing the eigenvalues
and eigenvectors of the tensor~$A$ as an intersection problem in
the $n$-dimensional weighted projective space 
\begin{equation*}
X \,\,\, = \,\,\, \PP(1,1,\ldots, 1, m-2).
\end{equation*}
The textbook definition of~$X$ can be found, for example, in \cite[\S 2.0]{CS}
and \cite[page 35]{fulton}. Points in~$X$ are represented by
vectors of complex numbers $(u_1: \cdots: u_n:
\lambda)$, not all zero, modulo the rescaling $(t u_1:\cdots: t u_n: t^{m-2}
\lambda)$ for any non-zero complex number~$t$.
The corresponding algebraic representation of our weighted projective space is
$\, X = \operatorname{Proj}(R)$, where $R = \CC[x_1, \ldots, x_n,
\lambda]$  is the polynomial ring with $x_1,\ldots,x_n$ having degree~$1$ 
and $\lambda$ having degree~$m-2$. 
The following proof uses basic toric intersection theory as in
\cite[Ch.~5]{fulton}.

\begin{proof}[Proof of Theorem \ref{thm:count}]
For $m = 2$, the expression  $((m-1)^n-1)/(m-2)$ simplifies to $ n$, which is
the number of eigenvalues of an ordinary $n {\times} n$-matrix. 
Hence we shall now assume that $\,m \geq 3$. For a fixed tensor~$A$, the
$n$ equations determined by $\,Ax^{m-1} = \lambda x\,$ correspond to $n$
homogeneous polynomials of degree $m-1$ in
our graded polynomial ring $R$.

Since $R$ is generated in degree~$m-2$, the line bundle $\mathcal O_X(m-2)$ 
is very ample. The corresponding lattice polytope~$\Delta$ is an 
$n$-dimensional simplex with vertices at $(m-2)e_i$ for $1
\leq i \leq n$, and $e_{n+1}$, where the $e_i$ are the basis vectors in
$\RR^{n+1}$. The affine hull of~$\Delta$ is the hyperplane
$\,x_1+\cdots + x_n +   (m{-}2)\lambda = m{-} 2$. The normalized volume of this
simplex equals
\begin{equation}
\label{eq:simplexvolume}
 {\rm Vol}(\Delta) \,\, =  \,\, (m-2)^{n-1} .
 \end{equation}
The lattice polytope $\Delta$ is smooth, except at the 
vertex~$e_{n+1}$, where it is simplicial with index $m-2$.
Therefore, the projective toric variety $X$ is simplicial, with
precisely one isolated singular point
corresponding to the vertex~$e_{n+1}$. 
By \cite[p.~100]{fulton}, the variety $X$ has a  rational Chow ring 
$A^*(X)_{\mathbb Q}$, which we can use to compute intersection numbers
of divisors on~$X$.

Our system of equations $Ax^{m-1} = \lambda x$ consists of $n$~polynomials of
degree $m-1$ in~$R$. Let $D$ be the divisor class corresponding to
$\Osh_X(m-1)$, and let $H$ be the very ample divisor class corresponding to
$\Osh_X(m-2)$. The volume formula (\ref{eq:simplexvolume}) is equivalent to
$\,H^n = (m-2)^{n-1}\,$ in $A^*(X)_{\mathbb Q}$, and we compute the
self-intersection number of $D$ as the following rational number:
\begin{equation*}
D^n \quad = \quad
\left(\frac{m-1}{m-2} \cdot H\right)^{\!\! n} \quad = \quad
 \left(\frac{m-1}{m-2}\right)^{\!\! n} \! \cdot (m-2)^{n-1} 
 \quad = \quad  \frac{(m-1)^n}{m-2}.
\end{equation*}

From this count we must remove the trivial solution $\{x=0\}$ of 
$\,Ax^{m-1} = \lambda x$. That solution corresponds to the 
singular point $e_{n+1}$ on~$X$. Since that point has index $m-2$, the trivial solution
counts for $1/(m-2)$ in the intersection computation, as shown in \cite[p.~100]{fulton}.
Therefore the number of non-trivial solutions in $X$ is equal to
\begin{equation}
\label{nicenumber} D^n \, - \, \frac{1}{m-2} \quad \, = \, \quad 
\frac{(m-1)^n-1}{m-2}, 
\end{equation}
Therefore, when the tensor $A$ admits only finitely many
equivalence classes of eigenpairs, then
their number, counted with multiplicities, coincides with
the positive integer in (\ref{nicenumber}).

In Example~\ref{ex:diagonal} below we exhibit a tensor $A$ which attains
the upper bound (\ref{nicenumber}). For that $A$,
each solution $(x,\lambda)$ has multiplicity $1$. It follows
that the same holds for generic $A$.
\end{proof}

\begin{remark}
An alternative presentation of our proof is to perform the substitution $\lambda =
\tilde\lambda^{m-2}$ in the equations $Ax^{m-1} = \lambda x$. This makes
the system of equations homogeneous of degree~$m-1$. B\'ezout's theorem
says that there are generically $(m-1)^n$ solutions in the projective space $\PP^n$. If we
remove the trivial solution, this leaves $(m-1)^n - 1$ solutions. Assuming that
none of these have $\tilde\lambda = 0$, the orbits formed by
multiplying $\tilde \lambda$ by $e^{2 \pi i/(m-2)}$ each yield the same
value of $\lambda = \tilde \lambda^{m-2}$ and $x$. Thus, there are
$((m - 1)^n - 1)/(m - 2)$ classes.

The delicate point in such a proof would be to argue that the solution to $A x^{m-1} = \tilde \lambda^{m-2} x$ has multiplicity $m-2$
even when $\tilde \lambda = 0$. In effect, toric geometry conveniently
does the bookkeeping in the correspondence between
solutions to $A x^{m-1} = \lambda x$ and solutions to $A x^{m-1} = \tilde
\lambda^{m-2} x$.
\end{remark}

\begin{ex}\label{ex:diagonal}
Let $A$ be the diagonal tensor of order $m$ and size $n$ defined by
setting $A_{i i \ldots i} = 1$ and all other entries zero. An eigenpair 
$(\lambda,x)$ is a solution to the equations
\begin{equation}
\label{eq:binomials}
  x_i^{m-1} \,\,=\,\, \lambda x_i  \qquad \hbox{for $\,1 \leq i \leq n$.} 
 \end{equation}
All non-trivial solutions in $X = \PP(1,\ldots,1,m-2)$
satisfy $\lambda \not= 0 $. By rescaling, we can assume that $\lambda = 1$.
Fix the root of unity $\zeta = e^{2\pi i/(m-2)}$, and let
$S = \{0, \ldots, m-3, *\}$. For any string
$\sigma$ in $S^n$ other than the all $*$s string, we define $x_i =
\zeta^{\sigma_i}$ if
$\sigma_i$ is an integer and $x_i = 0$ if $\sigma_i = *$.  This 
defines $(m-1)^n - 1$ eigenpairs.
However, some of these are equivalent.
Incrementing each integer in our string by $1$ modulo $m-2$ corresponds to
multiplying our eigenvector by~$\zeta$. Thus, we have defined $((m-1)^n -
1)/(m-2)$ equivalence classes of eigenvalues and eigenvectors. 
These are all equivalence classes of solutions to (\ref{eq:binomials}). 

More generally, suppose that $A$ is a diagonal tensor with $A_{ii\ldots i}$ equal
to some non-zero complex number~$a_i$. Then the eigenpairs are similarly given by $\lambda
= 1$ and $x_i = a_i^{1/{m-2}}\zeta^{\sigma_i}$ or $x_i = 0$, as above, where
$a_i^{1/{m-2}}$ is a fixed root of~$a_i$. In particular, for generic $a_i$ all
$((m-1)^n - 1)/(m-2)$ eigenpairs  will have distinct normalized eigenvalues. \qed
\end{ex}

From Theorem~\ref{thm:count}, we get the following result guaranteeing the
existence of real eigenpairs.

\begin{corollary}\label{cor:real-eig}
If $A$ has real entries and either $m$ or $n$ is odd, then $A$ has a real
eigenpair.
\end{corollary}

\begin{proof}
When either $m$ or $n$ is odd, then one can check that the integer
$\,((m-1)^n - 1)/(m-2)\,$ in
Theorem~\ref{thm:count} is odd. This implies that $A$ has a real eigenpair by
\cite[Corollary 13.2]{fulton-intersection}.
\end{proof}

Corollary~\ref{cor:real-eig} is sharp, in the sense that
there exist real tensors with no real eigenpairs whenever
both $m$ and~$n$ are even.  We illustrate this in
the following example.

\begin{ex}
Let $m$ be even, $n=2$, and $A$ the $2 {\times} \cdots {\times} 2$
tensor which is zero except for the entries $a_{12\cdots2} = 1$ and
$a_{21\cdots1} = -1$. The eigenpairs of $A$ are the solutions to the
equations:
\begin{align*}
x_2^{m-1} &\,=\, \lambda x_1 \\
-x_1^{m-1} &\,=\, \lambda x_2.
\end{align*}
Eliminating $\lambda$, we obtain $x_1^m + x_2^m = 0$, which has no
non-zero real solutions for even~$m$.

For $n$ any even integer, let $B$ be the tensor whose $n/2$ diagonal $2\times
\cdots \times 2$ blocks are the tensor~$A$ above, and which is zero elsewhere. A
non-trivial eigenpair must be an eigenpair for at least one of the blocks, and
therefore cannot be real. Thus, $B$ has no real eigenpairs. \qed
\end{ex}

\section{Characteristic Polynomial and Singular Tensors}\label{sec:characteristic}

The {\em characteristic polynomial} $\phi_A(\lambda)$ 
of a generic tensor $A$ was defined as follows  in  \cite{NQWW, Qi}.
Consider the univariate polynomial in $\lambda $
that arises by eliminating the unknowns
$x_1,\ldots,x_n$ from the system of equations  $A x^{m-1} = \lambda x$ and $x \cdot x = 1$. 
If $m$ is even then this polynomial equals $\phi_A(\lambda)$.
If $m$ is odd then this polynomial has the form $\phi_A(\lambda^2)$, i.e.\ the
characteristic polynomial evaluated at $\lambda^2$.  With these definitions,
Theorem \ref{thm:count} implies the following:

\begin{corollary} \label{cor:charpol}
For a generic tensor $A$, the characteristic polynomial $\phi_A(\lambda)$ is
irreducible and has degree $((m-1)^n-1)/(m-2)$. Hence this is the number
of normalized eigenvalues.
\end{corollary}

For any particular  tensor $A$, the {\em characteristic polynomial} $\phi_A(\lambda)$ is
obtained by specializing the entries in the coefficients of the generic characteristic polynomial.
Ni {\it et al.} \cite{NQWW} expressed $\phi_A(\lambda)$
as a Macaulay resultant, which implies a formula as a ratio
of determinants. For the present work, we used Gr\"obner-based software
to compute the characteristic polynomials of various tensors.
It is tempting to surmise that all zeros of the characteristic polynomial 
$\phi_A(\lambda)$ are normalized eigenvalues of the tensor $A$.
This statement is almost true, but not quite. There is some subtle fine print, to be
 illustrated by Example \ref{ex:fineprint} below.

Qi \cite[Question 1]{Qi} asked whether the set of normalized eigenvalues of a tensor is either finite
or all of~$\CC$. We answer this question by showing a tensor where neither of
these alternatives holds:

\begin{ex} \label{ex:fineprint}
Consider the complex $2 \times 2 \times 2$ tensor $A$ whose nonzero entries are
\begin{equation*}
a_{111} = a_{221} \, =\, 1 \quad \hbox{and} \quad a_{112} = a_{222} \,=\, i  = \sqrt{-1} .
\end{equation*}
We claim that any complex number other than $0$ is a normalized eigenvalue of
$A$, but $0$ is not a normalized eigenvalue.
The equations for an eigenvalue and eigenvector of~$A$ are
$$ x_1^2 + ix_1 x_2  \,= \, \lambda x_1 \quad \hbox{and} \quad
x_1 x_2 + ix_2^2  \,= \, \lambda x_2. $$
For any $\lambda \neq 0$ we obtain a matching eigenvector
that also satisfies $\,x \cdot x = 1\,$ by taking
\begin{equation*}
x = \left( \frac{\lambda^2 + 1}{2\lambda}, \,\frac{\lambda^2 - 1}{2 i \lambda}
\right).
\end{equation*}
Hence  $\lambda$ is a normalized eigenvalue. However, if $\lambda
= 0$, then an eigenvector must satisfy
$$
x_1^2 + i x_1 x_2 \, =\, 0 \quad \hbox{and} \quad
x_1 x_2 + i x_2^2 \, = \, 0.
$$
These  imply that $\,x \cdot x  = x_1^2 + x_2^2 \,$ is zero, so
$\lambda = 0$ cannot be a normalized eigenvalue. \qed
\end{ex}

However, we have the following weaker statement:

\begin{proposition}\label{prop:normalized-eigen}
The set of normalized eigenvalues of a tensor is either finite or it
consists of all complex numbers in the complement of a finite set.
\end{proposition}

\begin{proof}
The set $\mathcal{E}(A)$ of normalized eigenvalues $\lambda$ of the tensor $A$ is defined by the condition
\begin{equation*} \exists \,x \in \CC^n  \,\,: \,\,A x^{m-1} = \lambda x \,\,\,
\hbox{and} \,\,\, x \cdot x = 1. \end{equation*}
Hence $\mathcal{E}(A)$ is the image
of an algebraic variety in $\CC^{n+1}$ under the projection $(x,\lambda) \mapsto \lambda$.
Chevalley's Theorem states that the image of an algebraic variety under a polynomial map
is constructible, that is, defined by a Boolean combination of polynomial equations
and inequations. We conclude that the set $\mathcal{E}(A)$
of normalized eigenvalues is a constructible subset of $\CC$. This
means that $\mathcal{E}(A)$ is either a finite set or the complement of a finite set.
\end{proof}

The relationship between the normalized eigenvalues  and the characteristic
polynomial is summarized in the following proposition.

\begin{proposition}
\label{prop:sing}
For a tensor~$A$, each of the following conditions implies the next:
\begin{enumerate}
\item The set $\,\mathcal E(A)$ of all normalized eigenvalues consists of all
complex numbers
\item The set $\,\mathcal{E}(A)$ is infinite.
\item The characteristic polynomial $\phi_A(\lambda)$ vanishes identically.
\end{enumerate}
\end{proposition}

\begin{proof}
Clearly, (1) implies (2).
By the projection argument in the proof above,
the zero set in $\CC$ of the characteristic polynomial $\phi_A(\lambda)$
contains the set $\mathcal{E}(A)$. Hence (2) implies (3).
\end{proof}

From Example~\ref{ex:fineprint}, we see that
(2) does not necessarily imply (1), and in Example~\ref{ex:dreizwei}, we shall see
that (3) does not imply (2).
Qi defines a singular tensor to be one for which the first statement of
Proposition~\ref{prop:sing} holds. However, we suggest that the last condition
is a better definition: a 
\emph{singular tensor} is a tensor such that $\phi_A(\lambda)$ vanishes
identically. This definition has the advantage that the limit of singular
tensors is again singular. In particular, 
the set of all singular tensors is
a closed subvariety in the $n^m$-dimensional tensor space
$\CC^{n  \times \cdots \,\times n}$. Its defining polynomial equations
are the coefficients of the characteristic polynomial $\phi_A(\lambda)$
where $A$ is the tensor whose entries
$a_{i_1 \cdots i_n}$ are indeterminates.

\begin{example}
\label{ex:dreizwei}
Let $m = 3$ and $n=2$. Here $A = (a_{ijk})$ is a general tensor of format $2 \times 2 \times 2$.
The characteristic polynomial $\phi_A$ is obtained by eliminating $x_1$ and $x_2$ from
the ideal
$$  \langle \,
 a_{111}  x_1^2 + (a_{112} + a_{121}) x_1 x_2 +  a_{122}  x_2^2 - \lambda  x_1\, , \,
  a_{211}  x_1^2 + (a_{212} + a_{121}) x_1 x_2 +  a_{222}  x_2^2 - \lambda  x_2 \,, \,
  x_1^2 + x_2^2 - 1 \, \rangle .   $$
We find that $\phi_A$ has degree $3$, as predicted by Theorem \ref{thm:count}. 
Namely, the elimination yields
 $$ \phi_A(\lambda^2) \quad = \quad
 C_2 \lambda^6 \, + \,C_4 \lambda^4 \,+\, C_6 \lambda^2 + C_8, $$
 where $C_i$ is a certain homogeneous polynomial of degree $i$ in the
 eight unknowns $a_{ijk}$.
The set of singular $2 {\times} 2 {\times} 2$-tensors 
is the variety  in the projective space $\PP^7 = \PP(\CC^{2 \times 2 \times 2})$
given by 
\begin{equation}
\label{eq:ideal}
 \langle  C_2, C_4, C_6, C_8 \rangle \,\,\subset \,\, \CC[a_{111},a_{112}, \ldots,a_{222}] 
 \end{equation}
This is an irreducible variety of
 codimension~$2$ and degree~$4$, but the 
 ideal (\ref{eq:ideal}) is not prime.
 The constant coefficient $C_8$ is the square of a quartic. That quartic is 
   the {\em Sylvester resultant}
 $$ {\rm Res}_x (A x^2) \quad = \quad  {\rm det}
 \begin{pmatrix}
\, a_{111} & a_{112} + a_{121} & a_{122} & 0 \, \\
 \, 0 &  a_{111} & a_{112} + a_{121} & a_{122} \, \\
 \,  a_{211} & a_{212} + a_{221} & a_{222} & 0 \, \\
 \, 0 &  a_{211} & a_{212} + a_{221} & a_{222} \,
 \end{pmatrix}. \qquad
 $$
 The leading coefficient of the characteristic polynomial is a sum of squares
 $$ C_2  \, \,\,= 
\, \,\, (-a_{111} + a_{122} + a_{212} + a_{221} )^2 \,+\, ( a_{112} + a_{121} + a_{211} - a_{222})^2 $$
This indicates that among singular
$2 {\times} 2 {\times} 2$-tensors
those with real entries are scarce. Indeed,
the real variety of (\ref{eq:ideal}) is the union of two linear
spaces of codimension~$4$, with defining ideal
\begin{equation*}
\langle a_{122}, a_{211},
a_{112}+a_{121}-a_{222},
a_{212}+a_{221}-a_{111} \rangle \,\cap \,
\langle
a_{111} - a_{122},
a_{211} - a_{222},
a_{112} + a_{121},
a_{212} + a_{221} \rangle.
\end{equation*}
This explains why the singular tensor in Example \ref{ex:fineprint}
had to have a non-real coordinate.

We now look more closely at the real singular tensors defined by the second
ideal in this intersection. These are
tensors $A$ for which $\,Ax^2 = \bigl( a_{111}(x^2 + y^2) \,,\,
a_{211}(x^2 +y^2) \bigr)$. It is easy to see that, so long as $a_{111}$
and~$a_{211}$ are not both zero, the only normalized eigenvector is
\begin{equation*}
\left( \frac{a_{111}}{\sqrt{a_{111}^2 + a_{211}^2}},
\frac{a_{211}}{\sqrt{a_{111}^2 + a_{211}^2}} \right),
\mbox{ which has eigenvalue } \lambda = \sqrt{a_{111}^2 + a_{211}^2}.
\end{equation*}
In particular, the number of eigenvalues of such a tensor must be finite. 
This example shows that (3) does
not imply (2) in Proposition~\ref{prop:sing}.  \qed
\end{example}

The reader will not have failed to notice that  the notion of 
``singular'' used here (and in \cite{Qi}) is more restrictive
than the one familiar from the classical case $m=2$.
Indeed, a matrix is singular if it has $\lambda = 0$ as an eigenvalue,
or, equivalently, if the constant term of the characteristic polynomial
vanishes. That constant term is a power of the resultant
$\,{\rm Res}_x(A x^{m-1})$, and its vanishing means that
the homogeneous equations $A x^{m-1} = 0$ have a non-trivial
solution $x \in \PP^{n-1}$. This holds when the tensor $A$ is singular but not conversely.

Yet another possible notion of singularity for a tensor $A$ arises from its 
{\em hyperdeterminant} ${\rm Det}(A)$, as defined in \cite{gkz}.
For example, the hyperdeterminant of a $2 {\times} 2 {\times} 2 $-tensor equals
$$ \begin{matrix}
{\rm Det}(A) &=&
a_{122}^2 a_{211}^2 
+a_{121}^2 a_{212}^2
+a_{112}^2 a_{221}^2
+a_{111}^2 a_{222}^2 \qquad \qquad \qquad \qquad \qquad  \qquad
\qquad  \\ &  &
- 2 a_{121} a_{122} a_{211} a_{212}
 -2 a_{112} a_{122} a_{211} a_{221} -2 a_{112} a_{121} a_{212} a_{221}
 -2 a_{111} a_{122} a_{211} a_{222} \\ &  &
 - 2 a_{111} a_{121} a_{212} a_{222} - 2 a_{111} a_{112} a_{221} a_{222}
+4 a_{111} a_{122} a_{212} a_{221}  + 4 a_{112} a_{121} a_{211}
a_{222} .
\end{matrix}
$$
The hyperdeterminant vanishes if the hypersurface
defined by the multilinear form associated with $A$ has a singular point 
in $(\PP^{n-1})^m$. This property is unrelated to the
coefficients of the characteristic polynomial $\phi_A(\lambda)$.
In particular,  $\,{\rm Det}(A) \not= {\rm Res}_x(A x^{m-1}) $.

For an example take
the $2 {\times} 2 {\times} 2$-tensor $A$ all of whose entries are $a_{ijk} = 1$ is not singular
but ${\rm Det}(A) = 0$. On the other hand, the following tensor $A$ is
singular but has ${\rm Det}(A) =  -1$:
\begin{equation*} a_{111} = -1, \,
 a_{112} = 0, \,
  a_{121} = 0, \,
   a_{122} = -1, \,\,
a_{211} = 1,\,
 a_{212} = -1, \,
  a_{221} = 0, \,
   a_{222} = 1-i. \end{equation*}
This highlights the distinction between our setting here and that in
\cite[Proposition 2]{Lim}. 

\section{Dynamics on Projective Space}\label{sec:dynamics}

The purpose of this short section is to point out a connection to
dynamical systems. Dynamics on projective space is a well-established
field of mathematics \cite{BJ, ivashkovich}. 
We believe that the interpretation  of eigenpairs of tensors 
in terms of fixed points on $\PP^{n-1}$
could be of interest to applied mathematicians
as a new tool for modeling and numerical computations.

We consider the map $\psi_A$ defined by the formula $\psi_A(x) = Ax^{m-1}$. This is a
rational map from complex projective space $\PP^{n-1}$ to itself. The fixed
points of the map $\psi_A \colon \PP^{n-1} \dashrightarrow \PP^{n-1}$ are
exactly the eigenvectors of the tensor $A$ with non-zero eigenvalue, and the
base locus of $\psi_A$ is the set of eigenvectors with eigenvalue zero. In particular,
the map $\psi_A$ is defined everywhere if and only if $0$ is not an eigenvalue of~$A$.

Note that every such rational map arises from some tensor~$A$, but the tensor is
not unique.  Indeed, $A$ has $n^m$ entries while the map is
determined by $n$ polynomials, which have only $n \binom{n +m -2}{  m-1}$
distinct coefficients.  For instance, in Example~\ref{ex:dreizwei}, with
$m=3,n=2$, the eight entries of the tensor translate into six distinct
coefficients of the two binary quadrics that specify the self-map
of the projective line $\,\psi_A\colon \PP^1
\dashrightarrow \PP^1$. 

It is instructive to revisit the classical case $m=2$, where
$\psi_A\colon \PP^{n-1} \dashrightarrow \PP^{n-1}$ is a linear map.
The condition that every eigenvalue of the matrix~$A$ is zero
is equivalent to saying that
$A$ is {\em nilpotent}, that is, some matrix power of~$A$ is zero.
Geometrically, this means that some iterate of the rational map $\psi_A$ is
defined nowhere in projective space $\PP^{n-1}$.
We use the same definition for tensors: A is {\em nilpotent} if
some iterate of $\psi_A$ is nowhere defined.

\begin{proposition}
If the tensor $A$ is nilpotent then $0$ is the only eigenvalue of $A$.
The converse is not true:
there exist tensors with only eigenvalue $0$ that are not nilpotent.
\end{proposition}

\begin{proof}
Suppose $\lambda \not= 0 $ is an eigenvalue and $x \in \CC^n\backslash \{0\}$
a corresponding eigenvector. Then $x$ represents a point in $\PP^{n-1}$
that is fixed by $\psi_A$. Hence it is fixed by every iterate $\psi_A^{(r)}$ of $\psi_A$.
In particular, $\psi_A^{(r)}$ is defined at (an open neighborhood) of 
$x \in\PP^{n-1}$, and $A$ is not nilpotent.

Let $A$ be the $2 {\times} 2 {\times} 2$-tensor with 
$a_{111} = a_{211} = a_{212} = 1$ and the other five entries zero.
The eigenpairs of $A$ are the solutions to $\,
x_1^2 \, = \, \lambda x_1 \,$ and $\,
x_1^2 + x_1x_2  \, = \, \lambda x_2 $.
Up to equivalence, the only eigenpair is $x = (0, 1)$ and $\lambda = 0$.
However, the self-map $\psi_A$ on $\PP^1$ is dominant.
To see this, note that $\psi_A$ acts by translation on the affine line $\,\AAA^1 = \{
x_1 \neq 0\}$ because $(x_1^2: x_1^2 + x_1x_2) = (x_1:
x_1+x_2)$.  All iterates of $\psi_A$ are defined  on 
$\AAA^1$, i.e.~there are no base points with $x_1 \not= 0$, and hence $A$ is
not nilpotent.
\end{proof}

The example in the previous proof works because the two binary quadrics in $Ax^{m-1}$
have $x_1$ as a common factor. Indeed, whenever $n=2$, an eigenvector~$x$ has eigenvalue
zero if and only if $x$ is a solution to a common linear factor of the two binary forms
of $Ax^{m-1}$.

However, for $n \geq 3$, this is no longer true. 
Work in dynamics by Ivashkovic \cite[Theorem~1]{ivashkovich}
 implies that, for $n=3$, one can construct tensors~$A$ such that zero
is the only eigenvalue, the polynomials in $Ax^{m-1}$ have no common factors,
but $A$ is not nilpotent.

\begin{example}
This example is taken from \cite[Example~4.1]{ivashkovich}.
Let $m=n=3$ and take $A$ to be any tensor whose corresponding map
is the Cremona transformation
$$ \psi_A \colon \PP^2 \dashrightarrow \PP^2 \,:\,
 (x_1,x_2,x_3) \mapsto (x_1 x_2, x_1 x_3, 2 x_2 x_3).$$
This map has no fixed points, but it is not
nilpotent. The base locus of $\psi_A$ consists of the three
points  $(1,0,0)$, $(0,1,0)$, and~$(0,0,1)$, and, up to scaling, these are the
only eigenvectors of~$A$, all with eigenvalue~$0$.
\qed
\end{example}

\section{Symmetric Tensors}\label{sec:symmetric}

Of particular interest in numerical multilinear algebra is the situation 
when the tensor $A$ is symmetric and has real entries.
Here $A$ being {\em symmetric} means that 
the entries $a_{i_1 i_2 \cdots i_n}$ are invariant under
permuting the $n$ indices $i_1, i_2,\ldots,i_n$.
Each symmetric $n {\times} \cdots {\times} n$ tensor $A$ 
of order~$m$ corresponds to a unique homogeneous
polynomial~$f(x)$ of degree~$m$ in~$n$ unknowns.

The symmetric case is of interest because a real polynomial~$f(x)$ of even
degree~$m$ is positive semidefinite if and only if every real eigenpair of the
corresponding symmetric tensor~$A$ has non-negative
eigenvalue~\cite[Theorem~5(a)]{Qi-sym}.  This is illustrated in
Example~\ref{ex:motzkin}.

In the notation of \cite{Qi-sym}, the relation between the tensor and the
polynomial is written as
\begin{equation}
\label{eqn:iswrittenas} Ax^m \,\,=\,\, mf(x) \quad \hbox{and} \quad
A x^{m-1} \,\, =\,\, \nabla f(x) ,
\end{equation}
where $Ax^m$ is defined to be
\begin{equation*}
\sum_{i_1 = 1}^n \sum_{i_2 = 1}^n \cdots \sum_{i_m = 1}^n
a_{i_1 \ldots i_m} x_{i_1}
x_{i_2} \cdots x_{i_m} = x \cdot Ax^{m-1}.
\end{equation*}
The first equation in~(\ref{eqn:iswrittenas}) follows from the second because $x
\cdot \nabla f(x) = mf(x)$.
Note that the second equation in~(\ref{eqn:iswrittenas}) says that the
coordinates of the gradient of $f(x)$ are precisely the entries of
$A x^{m-1}$. 
The gradient $\nabla f(x)$ vanishes at a point $x$ in
$\PP^{n-1}$ if and only if $x$ is a singular point
of the hypersurface in $\PP^{n-1}$ defined by
the polynomial $f(x)$. This implies:

\begin{corollary} \label{cor:sing2}
The singular points of the projective hypersurface $\{x \in  \PP^{n-1}:f(x) = 0\}$ 
are precisely the eigenvectors of the corresponding
symmetric tensor $A$ which have eigenvalue~$0$.
\end{corollary}

The other eigenvectors  of $A$ can also be characterized in terms of 
the polynomial $f(x)$.

\begin{proposition}
Fix a non-zero $\lambda$ and suppose $m \geq 3$. Then $x \in \CC^n$ is a
normalized
eigenvector with eigenvalue $\lambda$ if and only if
$x$ is a singular point of the affine hypersurface
defined by the  polynomial
\begin{equation}\label{eqn:eig-char}
f(x) - \frac{\lambda}{2} x \cdot x - \left(\frac{1}{m} -
\frac{1}{2}\right) \lambda.
\end{equation}
\end{proposition}

\begin{proof}
The gradient of~(\ref{eqn:eig-char}) is $\nabla f - \lambda x = Ax^{m-1} -
\lambda x$, so every singular point $x$ is an eigenvector with
eigenvalue~$\lambda$. Furthermore, if we substitute $f(x) = \frac{1}{m} x \cdot \nabla f = 
\frac{\lambda}{m} x\cdot x\,$ into (\ref{eqn:eig-char}), then we obtain $x\cdot x = 1$.
This argument is reversible: if $x$ is a normalized eigenvector of $A$ then 
$x \cdot x = 1$ and $\nabla f(x)  = \lambda x$, and this implies that
(\ref{eqn:eig-char}) and its derivatives vanish.
\end{proof}

\begin{corollary} The characteristic polynomial $\phi_A(\lambda)$ is
a factor of the discriminant of (\ref{eqn:eig-char}).
\end{corollary}

Here we mean the classical multivariate discriminant \cite{gkz} of
an inhomogeneous polynomial of degree $m$ in $n$ variables $x$
evaluated at (\ref{eqn:eig-char}), where $\lambda$ is regarded
as a parameter. Besides the characteristic polynomial $\phi_A(\lambda)$, this
discriminant may contain other irreducible factors.

\begin{example}[\em Discriminantal representation of the characteristic polynomial
of a symmetric tensor]
If $n=2$ and $m=3$ then the discriminant at
bivariate cubic in  (\ref{eqn:eig-char}) equals
 $\,\lambda^4 \cdot \phi_A(\lambda)$.
If $n=2$ and $m=4$ then we evaluate the discriminant
of the ternary quartic using Sylvester's formula \cite[\S 3.4.D]{gkz}.
The output has the discriminant of
binary quartic as an extraneous factor:
$$ \hbox{\rm Discriminant of (\ref{eqn:eig-char})} \quad = \quad
 \bigl(\phi_A(\lambda) \bigr)^2  \cdot \lambda^9 \cdot
 \hbox{\rm Discriminant of $f(x)$}
 $$
 It would be interesting to determine the analogous factorization
 for arbitrary $m$ and $n$. \qed
\end{example}

\smallskip

The subject of this paper is the number of normalized eigenvalues of a tensor.
In Section~2 we gave an upper bound for that number under the
hypothesis that the number is finite. Remarkably, this 
hypothesis is not needed if we restrict our attention to 
symmetric tensors.

\begin{theorem}\label{thm:finite-norm-eig}
Every symmetric tensor $A$ has at most
$((m-1)^n-1)/(m-2)$ distinct normalized eigenvalues.
This bound is attained for generic symmetric tensors $A$.
\end{theorem}

\begin{proof}
It suffices to show that the number of normalized eigenvalues of
\underbar{every} symmetric tensor $A$ is finite. 
Recall from the proof of Theorem~\ref{thm:count} that the set of eigenpairs is
the intersection of $n$ linearly equivalent divisors on a
weighted projective space. Since these divisors are
ample, each connected component of the set of eigenpairs contributes at least
one to the intersection number. Therefore, the number of connected
components of eigenpairs can be no more than $((m-1)^n - 1)/(m-2)$. We conclude
that the number of normalized eigenvalues of $A$, if finite, must be bounded
above by that quantity as well.  Finally, Example~\ref{ex:diagonal} shows that
the bound is tight.  

We now prove that the number of normalized eigenvalues of a symmetric tensor $A$
is finite.
Let $S$ be the affine hypersurface in $\CC^n$ defined by
the equation $x_1^2 + \cdots + x_n^2 = 1$.
We claim  that a point $x \in S$ is an eigenvector of~$A$ if and
only if $x$ is a critical point of $f$ restricted to~$S$, in which case, the
corresponding eigenvalue $\lambda$ equals $\frac{1}{m}f(x)$.
By definition, a point $x \in S$ is  a critical point of $f \vert_S$ if and only
if the gradient $\nabla (f \vert_S)$ is zero at~$x$. The latter condition is
equivalent
to the gradient $\nabla f$ being a multiple of $\nabla(x_1^2 + \cdots + x_n^2 -
1) = 2x$. This is exactly the definition of an eigenvector.
Finally, if $x\in S$ is a critical point of $f \vert_S$, then $mf(x) = x \cdot \nabla f(x) = \lambda x
\cdot x = \lambda$, and hence $\,\lambda = \frac{1}{m}f(x)$.

Finally, to prove Theorem~\ref{thm:finite-norm-eig}, we note that, by generic
smoothness~\cite[Cor. III.10.7]{Hartshorne}, a polynomial function on a smooth 
variety has only finitely many critical values. 
Equivalently, Sard's theorem
in differential geometry says that the set of critical values of a
differentiable function has measure zero, so, by
Proposition~\ref{prop:normalized-eigen}, that set must be finite.
\end{proof}

We note two subtleties about Theorem~\ref{thm:finite-norm-eig}. First, it does
not imply that the characteristic polynomial of every symmetric tensor is
non-trivial. Second, the result is intrinsically tied to the normalization $x
\cdot x = 1$. We begin with an example of the first.

\begin{example}
Let $A$ be the symmetric $2 \times 2 \times 2$ tensor with
\begin{equation*}
a_{111} = -2i \,,\quad a_{112} = a_{121} = a_{211} = 1 \,,\quad a_{122} = a_{212} =
a_{221} = 0 \,, \quad a_{222} = 1.
\end{equation*}
Then, up to equivalence, the only eigenvectors are $(0, 1)$ with eigenvalue~$1$
and
$(1, i)$ with eigenvalue~$0$. Note that the second
cannot be rescaled to be a normalized eigenvector, so the only normalized
eigenvalue is~$1$.
However, the characteristic polynomial of~$A$ is identically zero. The reason is
that, for a small perturbation of $A$, the perturbation of the
eigenvector~$(1,i)$ can take on any given normalized eigenvalue. \qed
\end{example}

No analogue of Theorem~\ref{thm:finite-norm-eig} is possible with the
alternative normalization of requiring $x \cdot \overline x = 1$. In this case,
each equivalence class yields infinitely many eigenvalues, which nonetheless
have the same magnitude. However, the following example shows that the
magnitudes of the eigenvalues with $x \cdot \overline x = 1$ may still be an
infinite set.

\begin{example}
Let $A$ be the symmetric $3 {\times} 3 {\times} 3$ tensor whose non-zero entries are
\begin{equation*}
a_{111} = 2 \quad \hbox{and} \quad a_{122} = a_{212} = a_{221} 
\,=\, a_{133} = a_{313} = a_{331} \,=\, 1.
\end{equation*}
The eigenpairs of $A$ are the solutions to the equations
\begin{align*}
2 x_1^2 + x_2^2 + x_3^2 &\,=\, \lambda x_1, \\
2 x_1 x_2 &\,=\, \lambda x_2, \\
2 x_1 x_3 &\,=\, \lambda x_3.
\end{align*}
For any $\alpha \in \CC$, the vector $\,x = (1, i \alpha, \alpha)$
is an eigenvector with eigenvalue $\lambda = 2$. Rescaling,
$x/\sqrt{x \cdot \overline x}$ is an eigenvector with unit length and eigenvalue
\begin{equation*}
\frac{2}{\sqrt{1 + 2 \lvert \alpha \rvert}}.
\end{equation*}
The magnitude of this eigenvalue can be any real number in the interval
$(0,2\,]$.  Note that the  family of eigenvectors above all have $x \cdot x =
1$, so $\lambda = 2$ is the only normalized eigenvalue. \qed
\end{example}

One application of eigenvalues of symmetric tensors is 
that these can be used to decide whether a  polynomial $f$ is
{\em positive semidefinite}, i.e., whether $f(x) \geq 0$ for all $x \in \RR^n$.

\begin{example}\label{ex:motzkin}
The {\em Motzkin polynomial} $f(x,y,z) = z^6 + x^4 y^2 + x^2 y^4 - 3 x^2 y^2 z^2$ is a
well-known example of a positive semidefinite polynomial which cannot be written
as a sum of squares. Let $A$ be the corresponding 
$3 {\times} 3 {\times} 3 {\times} 3 {\times} 3 {\times} 3 $-tensor.
This tensor has $25$ eigenvalues, counting multiplicities,  six less
than our upper bound of $31$. Disregarding multiplicities, there are only
four distinct eigenvalues. All four are real
and they are equal to: $0$ (with multiplicity~$14$), $3/32$ (with multiplicity~$8$),
$3/2$ (with multiplicity~$2$), and $6$ (with multiplicity~$1$).
By~\cite[Theorem 5(a)]{Qi-sym}, this confirms the  fact 
that the Motzkin polynomial $f$ is positive semidefinite. \qed
\end{example}

\bigskip

{\bf Acknowledgments.}
We thank Tamara Kolda for inspiring this project, with a question she asked us
at the {\em Berkeley Optimization Day} on March 6, 2010.
Both authors were supported in part by the National Science Foundation
(DMS-0456960 and DMS-0757207).

\medskip
\end{document}